\documentclass[11pt]{amsart}

\usepackage[ansinew]{inputenc}
\usepackage{amsmath}
\usepackage{amsfonts}
\usepackage{amssymb}
\usepackage{amsthm}
\usepackage{url}

\input xy
\xyoption{all}

\newtheorem{thm}{Theorem}[section]
\newtheorem{lem}[thm]{Lemma}
\newtheorem{cor}[thm]{Corollary}
\newtheorem{prop}[thm]{Proposition}

\theoremstyle{definition}

\theoremstyle{remark}
\newtheorem{remark}[thm]{Remark}

\newcommand{\C}{\mathbb{C}}
\newcommand{\Z}{\mathbb{Z}}
\DeclareMathOperator{\id}{Id}
\DeclareMathOperator{\h}{h}
\DeclareMathOperator{\Dim}{dim}
\DeclareMathOperator{\Ker}{Ker}
\DeclareMathOperator{\Aut}{Aut}
\DeclareMathOperator{\Pic}{Pic}
\DeclareMathOperator{\Stab}{Stab}

\DeclareMathOperator{\PGL}{PGL}
\DeclareMathOperator{\Fix}{Fix}
\renewcommand{\P}{\mathbb{P}}
\renewcommand{\H}{H}

\newcommand{\CC}{\ensuremath{\mathbb{C}}}
\newcommand{\PP}{\ensuremath{\mathbb{P}}}
\newcommand{\ZZ}{\ensuremath{\mathbb{Z}}}

\DeclareMathOperator{\Sing}{Sing}

\newcommand{\OO}{\mathcal{O}}
\newcommand{\LL}{\mathcal{L}}
\newcommand{\ux}{\underline{x}}

\newcommand{\ls}[1]{\left | #1 \right |}
\newcommand{\set}[1]{\left \{ #1 \right \}}
\renewcommand{\span}[1]{\left < #1 \right >}

\newcommand{\Tr}{Tr}

\newenvironment{smallbmatrix}{\left [\begin{smallmatrix}}{\end{smallmatrix}\right ]}

\title[New examples of Calabi--Yau threefolds and genus zero surfaces]{New examples of Calabi--Yau threefolds\\ and genus zero surfaces}

\author{Gilberto Bini}
\address[Gilberto Bini]{Department of Mathematics, University of Milan, \hfill\newline\hfill \indent via Saldini 50,
I-20133 Milano, Italy}
\email{gilberto.bini@unimi.it}

\author{Filippo F. Favale}
\address[Filippo F. Favale]{Department of Mathematics, University of Pavia, \hfill\newline\hfill \indent via Ferrata 1,
I-27100 Pavia, Italy}
\email{filippo.favale@unipv.it}

\author{\\Jorge Neves}
\address[Jorge Neves]{CMUC, Department of Mathematics, University of Coimbra.\hfill\newline\hfill \indent Apartado 3008 - EC Santa Cruz,
3001-501 Coimbra, Portugal.}
\email{neves@mat.uc.pt}

\author{Roberto Pignatelli}
\address[Roberto Pignatelli]{Department of Mathematics, University of Trento,
\hfill\newline\hfill \indent via Sommarive 14,
I-38123 Trento, Italy}
\email{Roberto.Pignatelli@unitn.it}

\thanks{The third and fourth authors were partially supported by CMUC and FCT (Portugal) through the European program COMPETE/FEDER
and through projects PTDC/MAT/099275/2008 and PTDC/MAT/111332/2009. The first, second and fourth authors were partially supported by MIUR. The first and the fourth author were also supported by PRIN2012 "Geometria delle variet\`a algebriche" and by FIRB2012  "Moduli spaces and Applications".}


\begin{document}
\begin{abstract}
We classify the subgroups of the automorphism group of the product of $4$ projective lines admitting an
invariant anticanonical smooth divisor on which the action is free. As a first application,
we describe new examples of Calabi--Yau $3$-folds with small Hodge numbers. In particular, the Picard number is $1$ and the number of moduli is $5$. Furthermore, the fundamental group is non-trivial.
We also construct a new family of minimal surfaces of general type with geometric genus zero,  $K^2=3$ and fundamental group of order $16$. We show that this family dominates an irreducible component of dimension~$4$ of the moduli space of the surfaces of general type.
\end{abstract}
\maketitle


\section{Introduction}
\noindent A smooth ample divisor in a Calabi--Yau $3$-fold is a minimal surface of general type.
This simple observation yields a bridge between two important classes of algebraic varieties;
a bridge that has had many applications. The most famous example is the construction of the first Calabi--Yau
$3$-fold with nonabelian fundamental group by Beauville in \cite{beauville}, obtained by
extending Reid's construction in \cite{reid} of a Campedelli surface with fundamental group
isomorphic to the group of quaternions $Q_8$. Beauville shows that the surfaces constructed by
Reid are all ``rigid ample surfaces'' (i.e., smooth ample divisors $S$ in $Y$ such that $h^0({\mathcal O}_Y(S))=1$)
in the $3$-fold he constructs.

\noindent Beauville also points out that a rigid ample surface in a Calabi--Yau $3$-fold is a surface
with $p_g=0$, which is one of the most interesting classes of surfaces of general type. He also mentions that
whereas ``for surfaces with $p_g=0$ and $K^2=1$ or $2$ we have a great deal of information\,(...)\,little
is known about surfaces with $p_g=0$ and $K^2=3,4,5$.''  He refers to Inoue's examples with fundamental group
$Q_8 \oplus (\ZZ_2)^{K^2-2}$, asking if they are rigid ample surfaces in a Calabi--Yau 3-fold. \cite{AnnSNS} proves that the answer is affirmative when $K^2=3$.

\noindent Nowadays, we know a bit more on surfaces of general type with $p_g=0$, but
not that much. We know all possible algebraic fundamental groups of a minimal surface of
general type with $p_g=0$ and $K^2=1,2$, a short list of finite groups, and the
cases of bigger order are classified. But we know very little about the next cases.
For sure, a similar result is not possible for $K^2\geq 4$, since there are examples with infinite algebraic fundamental group:
see \cite{survey} for a more precise
account on the state of the art in this research area.

\noindent There is an old standing conjecture by Miles Reid
\cite[Conjecture 4]{Milespi1} which, by a result of Mendes Lopes and
Pardini \cite[Theorem 1.2]{MP07}, would imply that all surfaces of general type with $p_g=0$ and
$K^2=3$ have finite algebraic fundamental group. If the conjecture is true,
it should be possible to extend the results of the case $K^2 \leq 2$ to this case.
In particular, one could hope to classify the surfaces with the biggest possible fundamental groups.
There is a popular conjecture - not written anywhere, but usually attributed to M. Reid - asserting
that the maximal order should be $16$.

\noindent In the literature, there are $3$ families of such surfaces with fundamental group of cardinality $16$:
one with fundamental group $Q_8 \oplus \ZZ_2$ (\cite{Burniat3} and
\cite{AnnSNS}), one with fundamental group $\ZZ_2^2 \oplus \ZZ_4$ (\cite{mp}),
one constructed very recently (\cite{pg0new}) with fundamental group the
central product of the dihedral group with $8$ elements and $\ZZ_4$.
The first two families dominate irreducible components of the moduli space
(\cite{Burniat3}, \cite{Yifan}). In this paper, we construct a fourth family,
dominating an irreducible component.
More examples of surfaces with $p_g=0$ and $K^2=3$ have been recently constructed
in \cite{prodquot}, \cite{cartsteg}, \cite{PPS8}, \cite{PPS9}, \cite{KL} and \cite{Rito}. In all these examples 
the fundamental group is smaller or unknown.
\smallskip

\noindent On the Calabi--Yau side, physicists have focused recently on Calabi--Yau's with small Hodge
numbers  $(\h^{1,1},\h^{1,2})$: see, for instance, \cite{2}, \cite{4}, \cite{3}, \cite{6}, \cite{9} and \cite{BDS}.
In \cite{BF}, the authors describe some new examples of Calabi-Yau varieties. They are given as quotients of anticanonical sections of Fano varieties by finite groups $G$ acting freely. The Fano varieties are products of del Pezzo surfaces of various degrees. In particular, for the product $X$ of four complex projective lines, there exists a Calabi-Yau $Y$ with Hodge numbers 
$(\h^{1,1},\h^{1,2})=(1,5)$ and fundamental group isomorphic to $\Z_8\oplus\Z_2$.

\noindent In \cite{BF} an upper bound on the order of $G$ - depending on $X$ - has been found. This bound is maximal (equal to $16$) if and only if $X=\P^1\times \P^1\times \P^1\times \P^1$. It is then natural to ask which finite groups - among those of order $16$ - yield free and smooth quotients of Calabi-Yau manifolds. These quotients are again Calabi-Yau threefolds but have smaller height and non-trivial fundamental group. In this paper we investigate all possible actions, and come up with new non-isomorphic examples.

\smallskip
\noindent The main results of this work are the following. We contruct new families of Calabi--Yau
manifolds with small Hodge numbers. More precisely, we construct $4$ families of Calabi--Yau $3$-folds with fundamental
group of order $16$ and Hodge numbers $\h^{1,1}=1$, $\h^{1,2}=5$. We show that for two of these families, no
Calabi--Yau in the family contains a rigid ample divisor with $K^2=3$. On the contrary, in the other $2$
cases such a divisor exists, giving $2$ families of surfaces of general type with $p_g=0$.
One of these families is the family studied in \cite{AnnSNS}. The other family is a family of minimal surfaces of general type with $p_g=0$,
$K^2=3$ and fundamental group $\ZZ_4 \ltimes \ZZ_4$: there is no example of a surface with the same topological type in the literature.
We also show that this family dominates an irreducible component of the moduli space of the surfaces of general type.

\smallskip
\noindent The method is the following. Consider the $4$-fold $X=\P^1\times \PP^1\times \PP^1\times \PP^1$.
Assume that $G$ is a finite subgroup of $\Aut(X)$, and let $Y$ be a smooth divisor
in $|O_Y(2,2,2,2)|$, which is $G$-invariant and such that the action of $G$ on $Y$ is free.
Then the quotient $Y/G$ is also a smooth Calabi--Yau $3$-fold. In this situation we say that $(Y,G)$ is an admissible pair.
\smallskip

\noindent We classify all subgroups $G$ of $\Aut(X)$ which appear in an
admissible pair $(Y,G)$.  More precisely, we see that for each isomorphism class of groups, there is,
up to conjugacy, at most one possible subgroup of $\Aut(X)$, which may form an admissible pair.
We determine the groups which appear in an admissible pair, and we give
examples of admissible pairs in each case. The Hodge numbers of $Y/G$ depend only on $G$,
and we compute them in all cases.
\smallskip

\noindent This classification leads to exactly $4$ groups of order $16$.
In the two Abelian cases, we show that $G$ does not act on any divisor in $|O_X(1,1,1,1)|$;
in other words, $Y/G$ has no rigid ample divisor with $K^2=3$.
In the case of $G=\ZZ_4 \ltimes \ZZ_4$, such a divisor exists, yielding a $4$-dimensional
family of surfaces of general type with fundamental group $G$.
We show that this family dominates an irreducible component of the moduli space.
A similar result holds in the last case $G=\ZZ_2 \oplus Q_8$; we skipped this case because
that family was already completely studied in \cite{AnnSNS} and \cite{Burniat3}.


\section{Automorphisms of $X=\P^1\times \PP^1\times \PP^1 \times \PP^1$}

Every $g\in \Aut(X)$ acts on the $4$ factors (see, for instance, \cite{BF}) giving a  surjective homomorphism
$\pi\colon \Aut(X) \rightarrow S_4$ with kernel $\PGL(2)^{\times 4}$. On the other hand the permutations of the 
factors give an inclusion $S_4 \hookrightarrow \Aut(X)$ splitting $\pi$ and therefore giving a structure of semidirect product
$$
\Aut(X) \cong S_4 \ltimes \PGL(2)^{\times 4} 
$$
Concretely this gives, $\forall g \in \Aut X$, a unique decomposition $g=(A_i)\circ \sigma$ where $\sigma=\pi(g)$, 
$(A_i)=(A_1,A_2,A_3,A_4)\in \PGL(2)^{\times 4}$, and $\sigma(A_i)\sigma^{-1}=(A_{\sigma(i)})$.

If $g=(A_i)\circ \sigma$ and $h=(B_i)\circ\tau$,
$$h\circ g = (B_iA_{\tau(i)})\circ(\tau\circ\sigma),$$
$$h\circ g\circ h^{-1}= (B_iA_{\tau(i)}B_{(\tau \circ \sigma\circ \tau^{-1})(i)}^{-1})\circ (\tau\circ\sigma\circ\tau^{-1}).$$
%
\smallskip

\noindent For the purpose of what follows, we will denote by $A$ and $B$ the automorphisms of $\P^1$ that are represented,
respectively, by $$(x_0:x_1)\mapsto(x_0:-x_1)\quad\mbox{ and }\quad(x_0:x_1)\mapsto(x_1:x_0),$$ where
$(x_0:x_1)$ are projective coordinates on $\P^1$.
It is easy to see that $A$ and $B$ have order $2$. The group $\span{B}$ generated by $B$ is conjugated to the group $\span{A}$
generated by $A$ and this is true for
every subgroup of $\PGL(2)$ of order $2$. In general, the following holds true.
\begin{thm}[Klein]
\label{THM:KLEIN}
If $G$ is a finite subgroup of $\PGL(2)$, then $G$ is isomorphic to $\Z_n$, $D_{2n}$, $A_4$, $S_4$ or $A_5$.
Moreover, two isomorphic finite subgroups are conjugate.
\end{thm}

For every subgroup $G$ of $\Aut(X)$ we will denote by $\Fix(G)$ the set of the points of $X$ which are fixed by some nontrivial elements in $G$.

\begin{remark}\label{REM:DIMFIX} We are interested in automorphisms $g\in \Aut(X)$ whose fixed locus is
disjoint from the zero locus of a suitable section of $O_X(-K_X)$.
This implies \mbox{$\Dim\Fix\span{g}=0$}. In fact $-K_X$ is ample and therefore, if $C$ is a curve in
$\Fix\span{g}$, then $C\cdot Y >0$ so there is at least one fixed point of $\span{g}$ lying on $Y$.
On the other hand, $\Fix\span{g}$ is not empty by holomorphic Lefschetz fixed point formula.
\end{remark}
For the holomorphic fixed point formula we refer the reader to \cite[Theorem 4.12]{AB}. 
For the convenience of the reader, we recall here the following corollary which we will use.
\begin{thm}[Corollary of the holomorphic Lefschetz fixed point formula] 
Let $X$ be a compact complex manifold and let $f\in \Aut(X)$ be an automorphism with at most a finite number of isolated nondegenerate fixed points.
If we write $f^*|_{\H^{0,q}(X)}$ for the endomorphism induced by $f$ on $H^{0,q}(X)$, then
$$
\sum_{x\in \Fix(f)}\frac{1}{\det(I-d_xf)}=\sum_q (-1)^q \Tr(f^*|_{\H^{0,q}(X)}).
$$
\end{thm}
In our case, since for $X=(\P^1)^4$ one has $H^{0,*}(X)=H^{0,0}(X)\cong \mathbb{C}$, the right side
of the equation is equal to $1$. This implies that the left side has to be different from zero and, in particular, it is necessary to have at least a fixed point to have a contribution.

\smallskip 

It is shown in \cite{BF} that if $\Fix(G)$ does not intersect an anticanonical divisor,
then $|G|$ divides $16$. We have then $4$ cases to consider: groups of order $2$, $4$, $8$ or $16$.


\subsection{Subgroups of order $2$}\ \newline
The condition in Remark~\ref{REM:DIMFIX} is a very restrictive condition for an automorphism of order $2$.
In fact, up to conjugacy, there is only one
involution of $\Aut(X)$ with $0$-dimensional fixed locus, as the following lemma shows.

\begin{lem}
\label{LEM:ORDER2}
Assume that $g\in\Aut(X)$ is an automorphism of order $2$ such that $\Fix(g)$ has dimension $0$. Then $\pi(g)=\id$ and $g$ is conjugate to
$(A,A,A,A)$.
\end{lem}
\begin{proof}
If $g=(A_i)\circ \pi(g)$ has order $2$ then $\sigma:=\pi(g)$ has order at most $2$ and, since
$g^2=(A_iA_{\sigma(i)})\circ \sigma^2$, we have the relations $A_iA_{\sigma(i)} = \id$ in $\PGL(2)$. Up to conjugation, we can assume
$\sigma\in\{\id,(12),(12)(34)\}$. If $\sigma=(12)$, let $\underline{x}_3$ and $\underline{x}_4$ be two fixed points of
$A_3$ and $A_4$, respectively (the existence of which is guaranteed by the holomorphic Lefschetz fixed point formula).
Then, every point of the form
$$P_{\underline{x}_1}:=(\underline{x}_1,A_2\underline{x}_1,\underline{x}_3,\underline{x}_4),\quad \underline{x}_1\in\P^1,$$
is fixed. Indeed, since $A_1A_2=\id$, we get
$$P_{\underline{x}_1}\mapsto (A_1(A_2\underline{x}_1),A_2\underline{x}_1,A_3\underline{x}_3,A_4\underline{x}_4)
=((A_1A_2)\underline{x}_1,A_2\underline{x}_1,\underline{x}_3,\underline{x}_4)=P_{\underline{x}_1}.$$
Similarly, if $\sigma=(12)(34)$, the $2$-dimensional locus consisting of the points of the form
$(\underline{x}_1,A_2\underline{x}_1,\underline{x}_2,A_4\underline{x}_2)$
is pointwise fixed.

\vspace{2mm}

\noindent Accordingly, we may assume that $\sigma=\pi(g)=\id$, hence $g=(A_i)$ with (since $g^2=\id$) $A_i^2=\id$,
for all $i$. If $A_i=\id$ for some $i$,
one has at least a curve of fixed points (namely the $i$-th $\P^1$), so every $A_i$ has order $2$ as an
element of $\PGL(2)$. By Theorem \ref{THM:KLEIN}, all $A_i$ are conjugated in $\PGL(2)$ to $A$, so there exists $B_i\in \PGL(2)$ such that $B_i^{-1}A_iB_i=A$. Then  $(B_i)^{-1}\circ g\circ (B_i)=(A,A,A,A)$.
\end{proof}

\noindent Note that if $H$ is a nontrivial subgroup of $G$, then $\Fix(H) \subset \Fix(G)$.
In particular, if $\dim \Fix(G)=0$, every element of order $2$ in $G$ belongs to $\Ker(\pi)$.


\subsection{Subgroups of order $4$}\ \newline
Up to isomorphism, there are $2$ groups of order $4$, namely $\Z_4$ and $\Z_2\oplus\Z_2$.

\begin{lem}\label{LEM:Z2xZ2}
Assume that $\Z_2\oplus\Z_2\simeq G\leq \Aut(X)$ satisfies $\Dim \Fix(G)=0$. Then $G$ is
conjugated to $\span{g,h}$ with $g=(A,A,A,A)$ and $h=(B,B,B,B)$.
\end{lem}

\begin{proof}
Take two non-trivial elements $g,h\in G$. These have order $2$ and a finite number of fixed points.
By Lemma \ref{LEM:ORDER2}, we have $\pi(g)=\pi(h)=\id$ so $g=(A_i)$ and $h=(B_i)$. Since $g^2=h^2=ghg^{-1}h^{-1}=\id$, one obtains $A_iB_i=B_iA_i$
and $A_i^2=B_i^2=\id$, hence $\span{A_i,B_i}\leq \PGL(2)$ is isomorphic to a subgroup of $\Z_2\oplus\Z_2$. If this is a
proper subgroup of $\Z_2\oplus\Z_2$, then either $A_i,B_i$ or $A_iB_i$ is the identity and so one of $g$, $h$ or $gh$ has at
least a line of fixed points, which contradicts the assumptions. By Theorem
\ref{THM:KLEIN}, $\span{A_i,B_i}$ is conjugated to $\span{A,B}$ in $\PGL(2)$.
Moreover the internal automorphisms of $\PGL(2)$ which fix the subgroup $\span{A,B}$ act on $\{A,B,AB\}$
as the full group of permutations. Therefore there is a
$C_i\in \PGL(2)$ such that $C_i^{-1}A_iC_i=A$, $C_i^{-1}B_iC_i=B$;
setting $k:=(C_i)$, we have $k^{-1}\span{g,h}k=\span{(A),(B)}$.
\end{proof}

\begin{lem}
\label{LEM:ORDER4}
Assume that $g=(A_i)\circ \sigma\in \Aut(X)$ has order $4$ and that there exists an eigensection $s\in\H^0(X,O_X(-K_X))$
of $g$ such that $V(s)\cap\Fix\span{g}=\emptyset$. Then $g$ is conjugated to $(\id,A,\id,A)\circ (12)(34)$.
\end{lem}

\begin{proof}
Denote $G:=\span{g}\simeq \Z_4$. By Remark~\ref{REM:DIMFIX}, $\Dim\Fix(G)=0$. Since $o(g)=4$, we get
$\Fix(G)=\Fix(g^2)$, so $\Dim\Fix(g^2)=0$. Hence by Lemma~\ref{LEM:ORDER2}, $g^2$ is conjugated to $(A,A,A,A)$. Since
$$(A)=h^{-1}g^2h=(h^{-1}gh)^2,$$ we may assume, up to conjugation, that
$g^2=(A,A,A,A)$. Thus $g=(A_1,A_2,A_3,A_4)\circ \sigma$ with $\sigma^2=1$ and $A_iA_{\sigma(i)}=A$.
By conjugation, we may assume
$\sigma\in \{\id, (12),(12)(34)\}$. The fixed points of $G$ are $$\Fix(G)=\{(P_1,P_2,P_3,P_4) : P_i\in
\{(1:0),(0:1)\}\}.$$

\noindent
Let us now show that if $\sigma\in\{\id, (12)\}$, then at least a fixed point belongs to $V(s)$ for
every invariant section $s$. If $s$ is a section of $s\in\H^0(X,O_X(-K_X))$, then $s$ is a polynomial
of multidegree $(2,2,2,2)$ in the variables
$((x_{10},x_{11})$, $(x_{20},x_{21})$, $(x_{30},x_{31})$, $(x_{40},x_{41}))$.
The condition $V(s) \cap \Fix(G)=\emptyset$ is equivalent to all of the coefficients of the $16$ monomials
$x^2_{1i}x^2_{2j}x^2_{3k}x^2_{4l}$ being nonzero.
\smallskip

\noindent If $\sigma=\id$, then $A_j^2=A$ for all $j=1,\dots,4$. Hence $A_j=\begin{smallbmatrix}
1 & 0 \\
0 & \pm i
\end{smallbmatrix}$. Then $x^2_{10}x^2_{20}x^2_{30}x^2_{40}$ and $x^2_{10}x^2_{20}x^2_{30}x^2_{41}$ are eigenvectors
for the natural lift of the action of $g$ on $H^0(X,O_X(-K_X))$  with different eigenvalue.
So they cannot both appear with nontrivial coefficient in an eigensection $s$, a contradiction.

\smallskip

\noindent If $\sigma= (12)$, then $A_1A_2=A_2A_1=A_3^2=A_4^2=A$. Hence, we get $A_3,A_4\in \left\{\begin{smallbmatrix}
1 & 0\\
0 & \pm i
\end{smallbmatrix}\right\}$. If $k=(A_1,\id,\id,\id)\circ\id$, then
$$k^{-1}gk=(A_1^{-1}A_1\id,\id A_2A_1,A_3,A_4)\circ(12)=(\id,A,A_3,A_4)\circ(12),$$
so we may take $g$ of the form
$$g=\left(
\begin{smallbmatrix}
1 & 0 \\
0 & 1
\end{smallbmatrix},\begin{smallbmatrix}
1 & 0 \\
0 & -1
\end{smallbmatrix},\begin{smallbmatrix}
1 & 0 \\
0 & \pm i
\end{smallbmatrix},\begin{smallbmatrix}
1 & 0 \\
0 & \pm i
\end{smallbmatrix}\right)\circ (12).$$
As in the previous case, this leads to a contradiction, since $x^2_{10}x^2_{20}x^2_{30}x^2_{40}$ and $x^2_{10}x^2_{20}x^2_{30}x^2_{41}$ are eigenvectors for the natural lift of the action of $g$ on $H^0(X,O_X(-K_X))$ with different eigenvalue.
\medskip

\noindent We conclude that we must have $\sigma=(12)(34)$. Then $g=(A_1,A_2,A_3,A_4)\circ
(12)(34)$ with $g^2=(A)$.  Letting $k=(A_1,\id,A_3,\id)\circ\id$, we get
$k^{-1}gk=(\id,A,\id,A)\circ(12)(34)$.
\end{proof}


\subsection{Subgroups of order $8$}\ \newline
Up to isomorphism, there are $5$ groups of order $8$; these are $\Z_2^{\oplus 3}, D_8, \Z_8, \Z_2\oplus\Z_4$ and $Q_8$.

\begin{lem}\label{LEM:NOZ2^3}
No subgroup $G$ of $\Aut(X)$ with $\Dim\Fix(G)=0$ can be isomorphic to $\Z_2^{\oplus 3}$.
\end{lem}

\begin{proof}
Assume by contradiction that such a group $G\simeq \Z_2^{\oplus 3}$ exists.
By Lemma \ref{LEM:ORDER2}, $\pi(G)=\{\id \}$. Then $G$ is generated by $3$ elements of
order $2$ of the form $g_j=(A_{ij})$ for $j=1,2,3$. By Theorem \ref{THM:KLEIN},
the subgroup of $\PGL(2)$ generated by $A_{11},A_{12},A_{13}$ cannot be isomorphic to $\Z_2^{\oplus 3}$.
Thus $\span{A_{11},A_{12},A_{13}}$ must be isomorphic to one of the proper subgroups of $\Z_2^{\oplus 3}$.
We deduce that there exists a nontrivial $g \in G$ acting trivially on the first factor. Since every automorphism of $\PP^1$
has fixed points, the action of $g$ on the other two factors has some fixed points, giving a $1$-dimensional locus of fixed points of $g$ on $X$, contradicting $\Dim\Fix(G)=0$.
\end{proof}

\begin{lem}
\label{LEM:NOD8}
Let $G$ be a subgroup of $\Aut(X)$ isomorphic to $D_8$.
Then, for every eigensection $s\in\H^0(X,O_X(-K_X))$, we have $V(s)\cap\Fix(G) \neq \emptyset$.
\end{lem}

\begin{proof}
Assume by contradiction $V(s)\cap\Fix(G) = \emptyset$. Then by Remark~\ref{REM:DIMFIX} and
Lemma~\ref{LEM:ORDER2}, for every reflection $s \in G$, $\pi(g)=\id$.
By Lemma~\ref{LEM:ORDER4}, for a rotation $r$ of order $4$, $\pi(r) \not = \id$.
But in $D_8$ every rotation is product of two reflections, $r=s_1s_2$. This is a contradiction since then
$\id \neq \pi(r)=\pi(s_1)\pi(s_2)=\id \cdot \id$.
\end{proof}


\begin{lem}\label{LEM:Z4xZ2}
Assume that $G\simeq \Z_4\oplus \Z_2$ is a subgroup of $\Aut(X)$ such that there exists an eigensection $s\in\H^0(X,O_X(-K_X))$ with
$V(s)\cap\Fix(G) = \emptyset$. Then, up to conjugation, $G=\span{(\id,A,\id,A)\circ(12)(34), (B,B,B,B)}$.
\end{lem}
\begin{proof}
Assume that $G\simeq \Z_4\oplus \Z_2$ with generators $g$ and $h$.
By Lemma \ref{LEM:Z2xZ2}, we may assume $g^2=(A,A,A,A)$ and $h=(B,B,B,B)$.
By Lemma \ref{LEM:ORDER4}, $\pi(g)$ is conjugated to $(12)(34)$ so that we may assume $g=(A_i)\circ (12)(34)$.
Since
\[
A_1A_2=A_2A_1=A_3A_4=A_4A_3=A,
\]
and $A_iB=BA_i$, for all $i=1,\dots,4$ we deduce that $G$ is conjugated to a subgroup of $\Aut(X)$ given by
$\span{(\id,A,\id,A)\circ(12)(34), (B,B,B,B)}$, via $k:=(A_1,\id,A_3,\id)$.
\end{proof}


\begin{lem}\label{LEM:Z8}
Assume that $G\simeq \Z_8$ is a subgroup of $\Aut(X)$ such that there exists an eigensection $s\in\H^0(X,O_X(-K_X))$ with
$V(s)\cap\Fix(G) = \emptyset$. Then, up to conjugation, $G=\span{(\id,\id,\id,A)\circ(1324)}$.
\end{lem}
\begin{proof}
\noindent Choose a generator $g=(A_i)\circ \sigma$ of  $G\simeq \Z_8$. By Lemma \ref{LEM:ORDER4},
we may assume $g^2=(\id,A, \id,A)\circ(12)(34)$, so $\sigma$ is equal to $(1324)$ or its inverse. Substituting
$g$ with $g^{-1}$ we may assume $\sigma=(1324)$. Then $A_1=A_2=A_3^{-1}=AA_4^{-1}$, with $A_1A=AA_1$. We can then reduce to
$g=(\id,\id,\id,A)\circ(1324)$ by conjugation with $k=(A_1,A_1,\id,\id)$.
\end{proof}

\begin{lem}\label{LEM:Q8}
Assume that $G\simeq Q_8$ is a subgroup of $\Aut(X)$ such that there exists an eigensection $s\in\H^0(X,O_X(-K_X))$ with
$V(s)\cap\Fix(G) = \emptyset$. Then, up to conjugation,
$G=\span{(\id,A,\id,A)\circ (12)(34),(\id,A,A,\id)\circ(13)(24)}$.
\end{lem}
\begin{proof}

\noindent As usual for $G\simeq Q_8$, let $i$, $j$ and $k=ij$ be generators of order $4$.
By Lemma~\ref{LEM:ORDER4}, $\pi(i),\pi(j),\pi(k)\in \{(12)(34),(13)(24),(14)(23)\}$.
Since $\pi(i)\pi(j)=\pi(k)$, we deduce that $\pi(i),\pi(j)$ and $\pi(k)$ are all distinct.
We get \mbox{$\pi(Q_8)=\span{(12)(34),(13)(24)}\simeq \Z_2\oplus\Z_2$}. In particular, up to conjugation,
$i=(\id,A,\id,A)\circ (12)(34)$ and $j=(B_i)\circ (13)(24)$ with $j^2=(A,A,A,A)$. This forces $j$ to be of the form
$(B_1,B_2,B_1^{-1}A,B_2^{-1}A)\circ (13)(24)$ with $AB_i=B_iA$. Similarly, for
\mbox{$k:=ij=(B_2,AB_1,B_2^{-1}A,B_1^{-1})\circ (14)(23)$} we know that $k^2=(A,A,A,A)$, which implies $B_2=AB_1$.
Choosing $l=(B_1,B_1,\id,\id)$, we get
$l^{-1}il=i$ and $l^{-1}jl=(\id,A,A,\id)\circ(13)(24)$.
\end{proof}


\subsection{Subgroups of order $16$}\ \newline
Up to isomorphism, there are $14$ groups of order $16$.
These are:
\begin{small}
$$
\renewcommand{\arraystretch}{1.3}
\begin{array}{l}
\Z_8\oplus\Z_2,\quad \Z_4\oplus\Z_4,\quad \Z_{16},\quad \Z_4\oplus\Z_2\oplus\Z_2,\quad \Z_2\oplus\Z_2\oplus\Z_2\oplus\Z_2,\quad Q_{8} \oplus \Z_2,\quad D_{8}\oplus \Z_2,\\
\Z_4\ltimes \Z_4:=\span{g,h\,|\,g^4=h^4=\id,\, hgh^{-1}=g^{-1}},\quad Q_{16}:=\span{g,h,k\,|\,g^4=h^2=k^2=ghk}, \\
CP(D_8,Z_4):=\span{g,h,k\,|\,g^4=h^2=[g,h]=[h,k]=\id,\,g^2=k^2,\,ghg=h}, \\
SG(16,3):=\span{g,h,k\,|\,g^4=h^2=k^2=[g,h]=[h,k]=\id,\,kgk^{-1}=gh}, \\
D_{16}:=\span{g,h\,|\,g^8=h^2=\id,\, h^{-1}gh=g^{-1}}, \\
SD_{16}:=\span{g,h\,|\,g^8=h^2=\id,\, h^{-1}gh=g^3},\\
M_{16}:=\span{g,h\, |\,g^8=h^2=\id,\, h^{-1}gh=g^5}.
\end{array}
$$
\end{small}

\begin{prop}
Let $G$ be a subgroup of $\Aut(X)$ such that there is an eigensection $s\in\H^0(X,O_X(-K_X))$,
with $V(s)\cap\Fix(G) = \emptyset$.
Then $G$ cannot be isomorphic to any of the following groups:
$\Z_2^{\oplus 4}$, $\Z_4\oplus\Z_2^{\oplus 2}$, $SG(16,3)$, $CP(D_8,\Z_4)$,
$D_8\oplus \Z_2$, $D_{16}$, $SD_{16}$.
\end{prop}

\begin{proof}
This follows directly by Lemmas \ref{LEM:NOZ2^3} and \ref{LEM:NOD8}
since all the groups in the statement have a subgroup isomorphic either to $\Z_2^{\oplus 3}$ or to $D_8$.
\end{proof}

\begin{prop}
Let $G$ be a subgroup of $\Aut(X)$ such that there is an eigensection $s\in\H^0(X,O_X(-K_X))$,
with $V(s)\cap\Fix(G) = \emptyset$.
Then $G$ is not isomorphic to \mbox{$\Z_{16}$ or $Q_{16}$}.
\end{prop}
\begin{proof}
\noindent If $G\simeq \Z_{16}$, choose a generator $g$ and set $\sigma=\pi(g)$. By Lemma
\ref{LEM:ORDER4}, up to conjugation, $\sigma^4=\pi(g^4)=(12)(34)$. But $(12)(34)$ has no fourth
root in $S_4$, a contradiction.
If $G\simeq Q_{16}$ then, as in the list above, $g^4=h^2=ghgh$, with $g$ of order $8$ and $h,gh$ of order $4$. By the lemmas
\ref{LEM:ORDER4} and \ref{LEM:Z8}, we may assume $\pi(g)=(1234)$ whereas $\pi(h),\pi(gh)\in\{(12)(34),(13)(24),(14)(23)\}$.
This contradicts $\pi(g)\pi(h)=\pi(gh)$.
\end{proof}

\begin{prop}
Let $G$ be a subgroup of $\Aut(X)$ such that there is an eigensection $s\in\H^0(X,O_X(-K_X))$,
with $V(s)\cap\Fix(G) = \emptyset$. Then $G$ is not isomorphic to $M_{16}$.
\end{prop}

\begin{proof}
If $G$ is isomorphic to $M_{16}$, it has two generators $g$ and $h$ of order $8$ and $2$, respectively,
such that $h^{-1}gh=g^5$. By Lemma~\ref{LEM:ORDER2} and Lemma~\ref{LEM:Z8}, up to conjugation,
we can assume that $h=(A,A,A,A)$, and $g=(A_1,A_2,A_3,A_4)\circ \sigma$ where $\sigma=(1324)$.
Denote by $D$ the element $A_1A_2A_3A_4\in \PGL(2)$. Since $o(g)=8$, $D$ is an involution.
By direct computation, $h^{-1}gh=hgh=(AA_iA)\circ (1324)$ and $g^5=(D,D,D,DA_4)\circ (1324)$.
Since $h^{-1}gh=g^5$,
$$
AA_1A  = AA_2A  = AA_3A =  D\quad \text{and} \quad AA_4A  = DA_4.
$$
Therefore $A_1=A_2=A_3=ADA$ and then $D=A_1A_2A_3A_4=A_1^3A_4=ADAA_4$, so that $A_4=(AD)^2$.
Then, substituting in $AA_4A=DA_4$, we get $AADADA=DADAD$. Since $A^2=\id$ we get
$D=\id$. But then $g^4=(D,D,D,D)=\id$, contradicting $o(g)=8$.
\end{proof}

There are $4$ cases left, $\Z_8\oplus\Z_2$, $\Z_4\oplus\Z_4$,
$Q_8\oplus\Z_2$ and $\Z_4\ltimes \Z_4$. We will now show that they all occur and that in each
case the group action is determined up to conjugacy.

\begin{thm}\label{THM:4GPS}
Let $G$ be a subgroup of $\Aut(X)$ such that there is an eigensection $s\in\H^0(X,O_X(-K_X))$,
with $V(s)\cap\Fix(G) = \emptyset$. Then, up to conjugation,
\begin{enumerate}\renewcommand{\labelenumi}{(\arabic{enumi})}
\item $G=\span{(\id,\id,\id,A)\circ (1324), (B,B,B,B)} \simeq \ZZ_8 \oplus \ZZ_2$;
\item $G=\span{(\id,A,\id,A)\circ (12)(34), (\id,\id,B,B) \circ (13)(24)} \simeq \ZZ_4 \oplus \ZZ_4$;
\item $G=\span{(\id,A,\id,A)\circ (12)(34), (\id,A,B,AB) \circ (13)(24)}\simeq \Z_4\ltimes \Z_4$;
\item $G=\span{(\id,A,\id,A)\circ (12)(34), (\id,A,A,\id)\circ (13)(24), (B,B,B,B)}\simeq Q_8\oplus\Z_2$.
\end{enumerate}
\end{thm}
\begin{proof} By the previous propositions, $G$ is isomorphic to $\Z_8\oplus\Z_2$, $\Z_4\oplus\Z_4$,
$Q_8\oplus\Z_2$ or  $\Z_4\ltimes \Z_4$.
\smallskip

\paragraph{\bf Case (1)} Let $G\simeq \Z_8\oplus \Z_2$ with generators $g,h$ of order $8$ and $2$, respectively.
By Lemma \ref{LEM:Z8}, up to conjugacy, $g=(\id,\id,\id,A)\circ (1324)$. The element $h$ has order $2$ so it is of
the form $h=(B_1,B_2,B_3,B_4)$, with $B_i^2=\id$. Since $gh=hg$,
$$B_3=B_1, B_4=B_2, B_2=B_3 \mbox{ and } AB_1=B_4A,$$ so that $B_i\equiv B_1$ for all $i$, and $B_1AB_1^{-1}A^{-1}=B_1^2=\id$.
The elements of PGL(2) commuting with $A$ leave invariant the set of fixed points of $A$; they are represented
by a matrix which is either diagonal (if $B_1$ fixes both points) or antidiagonal (if $B_1$ exchanges them).
In the diagonal case, since $B_1^2=\id$, we get $B_1 \in \{\id, A\}$ and then $g^4=h$, a contradiction.

\noindent We deduce that $B_1$ is represented by an antidiagonal matrix
\[B_1=\begin{smallbmatrix}
0 & 1 \\
f^2 & 0
\end{smallbmatrix}
\]
for some $f\in \C^*$. Set $C=\begin{smallbmatrix}
1 & 0 \\
0 & f
\end{smallbmatrix}$ and $k:=(C,C,C,C)$. Then
\[k^{-1}gk=g\quad \text{and}\quad
k^{-1}hk=(B,B,B,B).
\]
\smallskip

\paragraph{\bf Case (2)} Let $G\simeq \Z_4\oplus \Z_4$. Let $g$ and $h$ be the two generators of the two
factors of $G$. We know that the subgroup $\span{g,h^2}\simeq \Z_4\oplus \Z_2$ is conjugated to
$$\span{(\id,A,\id,A)\circ (12)(34),(B,B,B,B)},$$ so we may assume $g=(\id,A,\id,A)\circ (12)(34)$ and $h^2=(B,B,B,B)$.
By Lemma \ref{LEM:ORDER4},
$\sigma:=\pi(h)\in\{(12)(34),(13)(24),(14)(23)\}$. Since the same is true for $\pi(gh)$, we have
$\sigma\neq(12)(34)$. By replacing $h$ with $gh$, we may assume $\sigma=(13)(24)$.
Since $gh=hg$ and $h^2=(B,B,B,B)$,
$$h=(C,C,C^{-1}B,C^{-1}B)\circ (13)(24)$$ with $CAC^{-1}A^{-1}=CBC^{-1}B^{-1}=\id$. Conjugation with $k:=(C,C,\id,\id)$
sends $g$ to $g$ and $h$ to $(\id,\id,B,B)\circ (13)(24)$.
\smallskip

\paragraph{\bf Case (3)}
Let $G\simeq \Z_4\ltimes \Z_4$, $G=\span{g,h\,|\,g^4=h^4=ghgh^{-1}=\id}$. Take $g$ and $h$ to be generators
of the two $\Z_4$. As with the case of $\Z_4\oplus\Z_4$, we may assume $g=(\id,A,\id,A)\circ (12)(34)$, $h^2=(B,B,B,B)$ and $\pi(h)= (13)(24)$.
Since $ghg=h$,
$$h=(C,AC,BC^{-1},ABC^{-1})\circ (13)(24)$$ with $CAC^{-1}A^{-1}=CBC^{-1}B^{-1}=\id$. Conjugation with
$k:=(C,C,\id,\id)$ sends $g$ to $g$ and $h$ to $(\id,A,B,AB)\circ (13)(24)$.
\smallskip

\paragraph{\bf Case (4)}
Let $G\simeq Q_8\oplus \Z_2$. By Lemma \ref{LEM:Q8}, we can assume that $Q_8=\span{i,j}$ with
$i=(\id,A,\id,A)\circ(12)(34)$ and $j=(\id,A,A,\id)\circ (13)(24)$. Let $h=(C_1,C_2,C_3,C_4)$ be the generator of the
$\Z_2$ factor. Since $ih=hi$ and $jh=hj$, $C_1=C_2=C_3=C_4=:C$ with $C$ commuting with $A$ and $C^2=\id$. These two
conditions are satisfied if $C=\id$, $C=A$ or $C$ is antidiagonal.
If $C=\id$ or $C=A$ then $h\in Q_8$ which is impossible.
Then $C=\begin{smallbmatrix}
0 & 1 \\
f^2 & 0
\end{smallbmatrix}$ for some $f\in \C^*$.  Let $D=\begin{smallbmatrix}
0 & 1 \\
f & 0
\end{smallbmatrix}$. Conjugation by $(D,D,D,D)$ fixes $i$ and $j$ and sends $h$ to $(B,B,B,B)$.
\end{proof}


\section{Heights and Hodge numbers for $Y/G$}

In what follows $G$ is a finite subgroup of $\Aut(X)$ and $Y$ is a smooth Calabi--Yau threefold in $X=\PP^1\times \PP^1\times \PP^1\times \PP^1$
such that $(Y,G)$ is an admissible pair, i.e., $G\leq \Stab_{\Aut(X)}(Y)$ and $\Fix(G)\cap Y=\emptyset$. We will
compute all pairs of Hodge numbers $(\h^{1,1},\h^{1,2})$ and relative heights $h:=\h^{1,1}+\h^{1,2}$ of $Y/G$. Every case is realized
as a partial quotient of the examples of section \ref{SECT:GOODGORDER16}. For the reader convenience we recall 
that, for every anticanonical divisor $Y$ in $X$, $\chi(Y)=-128$ (by, {\it e.g.}, \cite[Theorem 3.1]{BF}) and that for every Calabi-Yau $Z$,
$\chi(Z)=2(\h^{1,1}(Z)-\h^{1,2}(Z))$.
\vspace{2mm}

\noindent Recall that, for an admissible pair $(Y,G)$,
$\chi(Y/G)=\chi(Y)/|G|$ and
$$\H^{1,1}(Y/G)=\H^{1,1}(Y)^G=(\Pic(X)\otimes\C)^G=(\Pic(X)\otimes\C)^{\pi(G)}.$$
Hence it suffices to consider the image $\pi(G)$. For example, we have seen in Theorem \ref{THM:4GPS} that, if
$G\simeq \Z_4\ltimes \Z_4$ and $(Y,G)$ is an admissible pair then $\pi(G)=\span{(12)(34),(13)(24)}$. The action of $G$ on
$\C^4=\Pic(X)\otimes\C$ has an invariant space of dimension $1$ (the diagonal) so $\h^{1,1}(Y/G)=1$.
\vspace{4mm}

\noindent If $|G|=2$ one has $G\simeq \Z_2$, $\chi(Y/G)=-64$, $\pi(G)=\{\id\}$ and
\begin{itemize}
\item $(\h^{1,1}(Y/G),\h^{1,2}(Y/G))=(4,36)$ and $h(Y/G)=40$.
\end{itemize}


\noindent If $|G|=4$ then $\chi(Y/G)=-32$ and one of the following holds:
\begin{itemize}\setlength{\itemsep}{.2cm}
\item $G\simeq \Z_2\oplus\Z_2$, $\pi(G)=\{\id\}$, $(\h^{1,1}(Y/G),\h^{1,2}(Y/G))=(4,20)$ and $h(Y/G)=24$;
\item $G\simeq \Z_4$, $\pi(G)=\span{(12)(34)}$, $(\h^{1,1}(Y/G),\h^{1,2}(Y/G))=(2,18)$ and $h(Y/G)=20$.
\end{itemize}


\noindent If $G$ has order $8$, then $\chi(Y/G)=-16$ and one of the following holds:
\begin{itemize}\setlength{\itemsep}{.2cm}
\item $G\simeq \Z_8$  with $\pi(G)\simeq \span{(1234)}\simeq\Z_4$,\newline
$(\h^{1,1}(Y/G),\h^{1,2}(Y/G))=(1,9)$ and $h(Y/G)=10$;
\item $G\simeq Q_8$ with $\pi(G)=\span{(12)(34),(13)(24)}\simeq \Z_2\oplus\Z_2$,\newline
$(\h^{1,1}(Y/G),\h^{1,2}(Y/G))=(1,9)$ and $h(Y/G)=10$;
\item $G\simeq \Z_4\oplus\Z_2$ with $\pi(G)\simeq \span{(12)(34)}\simeq \Z_2$,\newline
$(\h^{1,1}(Y/G),\h^{1,2}(Y/G))=(2,10)$ and $h(Y/G)=12$.
\end{itemize}


\noindent Finally, if $|G|=16$, then $\chi(Y/G)=-8$ and each
quotient has the same Hodge numbers and height equal to $(1,5)$ and $6$, respectively. More precisely one of the following holds:
\begin{itemize}\setlength{\itemsep}{.2cm}
\item $G\simeq \Z_8\oplus \Z_2$  with $\pi(G)\simeq \span{(1234)}\simeq\Z_4$;
\item $G\simeq Q_8\oplus\Z_2$ with $\pi(G)=\span{(12)(34),(13)(24)}\simeq \Z_2\oplus\Z_2$;
\item $G\simeq \Z_4\oplus\Z_4$ with $\pi(G)\simeq \span{(12)(34),(13)(24)}\simeq \Z_2\oplus\Z_2$;
\item $G\simeq \Z_4\ltimes\Z_4$ with $\pi(G)\simeq \span{(12)(34),(13)(24)}\simeq \Z_2\oplus\Z_2$.
\end{itemize}

\section{Admissible pairs with group of order $16$}
\label{SECT:GOODGORDER16}
In this section we give examples of admissible pairs $(Y,G)$ with $G$ maximal, i.e., with $|G|=16$.
To be more precise, we give an example for each subgroup $G$ of $\Aut(X)$ of order $16$ that acts freely on a
Calabi--Yau threefold which is stabilized by $G$. As previously said, $Y/G$
have Hodge numbers $(1,5)$. By taking partial quotients, one obtains all the other examples.

\medskip

\subsection{Admissible pairs with $G=\Z_4\ltimes \Z_4$}\ \newline
By Theorem \ref{THM:4GPS}, up to conjugation  $G:=\left<g,h\right > \subset \Aut(X)$ with
\begin{equation}\label{eq: action by Z_4ltimesZ_4}
\renewcommand{\arraystretch}{1.5}
\begin{array}{l}
g:=\left(
\id,
A,\id,A,\right)\circ (12)(34),\\
h:=\left(\id,A,B,AB\right)\circ (13)(24).
\end{array}
\end{equation}
$G$ is a $2$-group with $3$ elements of order $2$, namely $g^2=\left(A,A,A,A\right)$,
$h^2=\left(B,B,B,B\right)$, $g^2h^2=\left(AB,AB,AB,AB\right)$. Therefore $\Fix (G)=\Fix(g^2) \cup \Fix(h^2) \cup \Fix(g^2h^2)$.
These are $3$ disjoint sets of $16$ points each:
\begin{equation}\label{FixedPoints}
\renewcommand{\arraystretch}{1.3}
\begin{array}{l}
\Fix(g^2)=\{(\ux_1,\ux_2,\ux_3,\ux_4) : \forall i\ \ux_i\in\set{(0,1),(1,0)}\},\\
\Fix(h^2)=\{(\ux_1,\ux_2,\ux_3,\ux_4) : \forall i\ \ux_i\in\set{(1,1),(1,-1)}\},\\
\Fix(g^2 h^2)=\{(\ux_1,\ux_2,\ux_3,\ux_4) : \forall i\ \ux_i\in\set{(1,i),(1,-i)}.\}
\end{array}
\end{equation}

\noindent
$H^0(\OO_X(2,2,2,2))$ has basis $\set{x_{10}^{a_1}x_{11}^{b_1}x_{20}^{a_2}x_{21}^{b_2}x_{30}^{a_3}x_{31}^{b_3}x_{40}^{a_4}x_{41}^{b_4} : a_1+b_1=\cdots=a_4+b_4=2}$.
Consider $Q_0,\dots,Q_5\in H^0(\OO_X(2,2,2,2))$ given by:
\begin{equation}\label{eq: generators of the linear system of quadrics}
\renewcommand{\arraystretch}{1.35}
\begin{array}{l}
Q_0 := \textstyle\prod_{i=1}^4 x_{i0}x_{i1}, \\
Q_1 := (x_{11}^2x_{20}^2 +x_{10}^2x_{21}^2)(x_{31}^2 x_{40}^2+ x_{30}^2x_{41}^2), \\
Q_2 := x_{20}x_{21}x_{30}x_{31}(x_{10}^2+x_{11}^2)(x_{40}^2+x_{41}^2)-x_{10}x_{11}x_{40}x_{41}(x_{20}^2+x_{21}^2)(x_{30}^2+x_{31}^2), \\
Q_3 := x_{10}x_{11}x_{30}x_{31}(x_{20}^2+x_{21}^2)(x_{40}^2+x_{41}^2)+x_{20}x_{21}x_{40}x_{41}(x_{10}^2+x_{11}^2)(x_{30}^2+x_{31}^2), \\
Q_4 := (x_{10}^2+x_{11}^2)(x_{20}^2+x_{21}^2)(x_{30}^2+x_{31}^2)(x_{40}^2+x_{41}^2),\\
Q_5 := (x_{10}^2x_{20}^2+x_{11}^2x_{21}^2)(x_{30}^2x_{40}^2+x_{31}^2x_{41}^2).
\end{array}
\end{equation}

\noindent
Denote by $\LL\subset \ls{\OO_X(2,2,2,2)}$ the linear system generated by $Q_0,\dots,Q_5$.

\begin{lem}\label{CYlem}
The linear system $\LL$ is $5$-dimensional and has $64$ base points, given by:
\[\renewcommand{\arraystretch}{1.3}
\begin{array}{l}
\set{x_{10}x_{11}=x_{20}x_{21}=x_{30}^2+x_{31}^2=x_{40}^2-x_{41}^2=0}\cup \\
\set{x_{10}x_{11}=x_{20}x_{21}=x_{30}^2-x_{31}^2=x_{40}^2+x_{41}^2=0}\cup \\
\set{x_{10}^2+x_{11}^2=x_{20}^2-x_{21}^2=x_{30}x_{31}=x_{40}x_{41}=0}\cup \\
\set{x_{10}^2-x_{11}^2=x_{20}^2+x_{21}^2=x_{30}x_{31}=x_{40}x_{41}=0}\cdot
\end{array}
\]
\end{lem}
\begin{proof}
It is not hard to see that the $Q_i$ are linearly independent, for example, by expanding $Q_0,Q_1,Q_2,Q_3,Q_4-Q_1-Q_5,Q_5$ and checking
that there is no common monomial in any two of them.
\medskip

\noindent
A base point of $\LL$ must satisfy $(x_{10}^2+x_{11}^2)(x_{20}^2+x_{21}^2)(x_{30}^2+x_{31}^2)(x_{40}^2+x_{41}^2)=0$.
Assume that $x_{10}^2+x_{11}^2=0$. Using $Q_0$ we get $x_{20}x_{21}x_{30}x_{31}x_{40}x_{41}=0$. Notice that $x_{20}x_{21}\not =0$ for,
otherwise, $Q_1=Q_2=Q_3=Q_5=0$ would reduce to
\[\renewcommand{\arraystretch}{1.3}
\begin{cases}
x_{31}^2x_{40}^2 + x_{30}^2x_{41}^2 = 0\\
x_{40}x_{41}(x_{30}^2+x_{31}^2)=0 \\
x_{30}x_{31}(x_{40}^2+x_{41}^2)=0 \\
x_{30}^2x_{40}^2 + x_{31}^2x_{41}^2 = 0
\end{cases}
\]
which is impossible. Then $x_{30}x_{31}x_{40}x_{41}=0$ and at least one among $x_{30}x_{31}$ and $x_{40}x_{41}$ vanish.
Suppose  $x_{30}x_{31}=0$.
Then $Q_2$ reduces to $x_{40}x_{41}(x_{20}^2+x_{21}^2)=0$. Arguing as above we can show that $x_{20}^2+x_{21}^2\not = 0$.
Hence also $x_{40}x_{41}=0$ and either $Q_1$ or $Q_5$ reduces to $x_{20}^2-x_{21}^2=0$.
Assuming  $x_{40}x_{41}=0$ instead of  $x_{30}x_{31}=0$ a similar argument leads to the same conclusion.
The statement follows by repeating the
same argument with starting assumption $x_{20}^2+x_{21}^2=0$, $x_{30}^2+x_{31}^2=0$ or $x_{40}^2+x_{41}^2=0$.
\end{proof}

\begin{cor}
The general $Y \in \LL$ is a Calabi--Yau $3$-fold on which $G$ acts freely.
\end{cor}
\begin{proof}
Notice that $G$ leaves invariant each of the divisors $(Q_i=0)$, $i=0,\dots,5$; indeed
$g$ and $h$ fix all of $Q_i$. In particular, the action of $G$
induces an action on every $Y \in \LL$. Moreover, the base locus of $\LL$ computed in Lemma \ref{CYlem}
does not contain any of the fixed points of the action, given in (\ref{FixedPoints}); hence the action
on the general $Y\in \LL$ is free.
\medskip

\noindent
Consider the points $[(\pm i,1),(\pm 1,1),(0,1),(0,1)]$ of the base locus of $\LL$. After localizing at
the affine open set of $X$ given by $x_{11}=x_{21}=x_{31}=x_{41}=1$, the equation of $Q_4$ becomes
$(x_{10}^2+1)(x_{20}^2+1)(x_{30}^2+1)(x_{40}^2+1)$.
It is easy to see that this equation defines a hypersurface in $\CC^4$ smooth at the points $(\pm i,\pm 1,0,0)$.
Similarly, one checks that $(Q_4=0)$ is smooth at all of the $64$ base points in Lemma \ref{CYlem}.
By Bertini's Theorem, it follows that the general $Y\in \LL$ is smooth.
By Lefschetz Hyperplane Section Theorem, $Y$ is simply connected and since
$\omega_Y=\omega_X(2,2,2,2)_{|Y}={\mathcal O}_Y$, $Y$ is a Calabi--Yau.
\end{proof}

\subsection{Admissible pairs with $G=\Z_4\oplus \Z_4$}\ \newline
By Theorem \ref{THM:4GPS}, up to conjugation, $G:=\left <g,h\right >$ with
\begin{equation}\label{eq: action by Z_4xZ_4}
\renewcommand{\arraystretch}{1.5}
\begin{array}{l}
g:=\left( \id,A,\id,A \right)\circ (12)(34),\\
h:=\left(\id,\id,B,B\right)\circ (13)(24)
\end{array}
\end{equation}
The following $6$ homogeneous forms
\[\renewcommand{\arraystretch}{1.35}
\begin{array}{l}
Q'_0 := \textstyle\prod_{i=1}^4 x_{i0}x_{i1}, \\
Q'_1 := (x_{11}^2x_{20}^2 +x_{10}^2x_{21}^2)(x_{31}^2 x_{40}^2+ x_{30}^2x_{41}^2), \\
Q'_2 := x_{20}x_{21}x_{30}x_{31}(x_{10}^2-x_{11}^2)(x_{40}^2-x_{41}^2)-x_{10}x_{11}x_{40}x_{41}(x_{20}^2-x_{21}^2)(x_{30}^2-x_{31}^2), \\
Q'_3 := x_{10}x_{11}x_{30}x_{31}(x_{20}^2+x_{21}^2)(x_{40}^2+x_{41}^2)+x_{20}x_{21}x_{40}x_{41}(x_{10}^2+x_{11}^2)(x_{30}^2+x_{31}^2), \\
Q'_4 := (x_{10}^2+x_{11}^2)(x_{20}^2+x_{21}^2)(x_{30}^2+x_{31}^2)(x_{40}^2+x_{41}^2),\\
Q'_5 := (x_{10}^2x_{20}^2+x_{11}^2x_{21}^2)(x_{30}^2x_{40}^2+x_{31}^2x_{41}^2).
\end{array}
\]
are invariant under the natural action of $G$ on the vector space $\H^0(O_X(2,2,2,2))$. As in the case of $\ZZ_4\ltimes \ZZ_4$, it can
be shown that these forms span a $5$-dimensional linear system whose general member is smooth, has trivial canonical bundle, is simply connected and on which $G$ acts without fixed points.

\subsection{Admissible pairs with $G=\Z_8\oplus \Z_2$}\ \newline
By Theorem \ref{THM:4GPS}, up to conjugation, $G:=\left <g,h\right >$ with
\begin{equation}\label{eq: action of Z_8xZ_2}
\renewcommand{\arraystretch}{1.5}
\begin{array}{l}
g:=\left(\id,\id,\id,A\right)\circ (1324),\\
h:=\left(B,B,B,B\right)
\end{array}
\end{equation}
An explicit Calabi--Yau threefold $Y\subset X$, member of $\ls{\OO_X(2,2,2,2)}$, invariant under the action by $G$ and on which $G$ acts without fixed points
is given in \cite{BF}.

\subsection{Admissible pairs with $G=Q_8\oplus \Z_2$}\ \newline
For this example we refer to \cite{AnnSNS}, where it is given an action of $G$ on $X$
which is not exactly the one described in Theorem \ref{THM:4GPS}, but a conjugated of it. Indeed
in \cite[Theorem 1.1]{AnnSNS} there is a family of hypersurfaces $Z_2 \in |O_X(-K_X)|$ whose general
element is smooth and such that the group acts freely on it. By Theorem  \ref{THM:4GPS}, that action is
conjugated to the one given here, and $(Z_2,G)$ is an admissible pair.

\section{Surfaces of general type}

In this section we consider the families of admissible pairs $(Y,G)$ of a Calabi--Yau $3$-fold $Y\in \ls{\OO_X(2,2,2,2)}$
and a group of order $16$, acting on $X=\PP^1\times \PP^1\times \PP^1 \times \PP^1$ such that $Y$ is invariant
under this action and does not meet the fixed locus of $G$ on $X$. The aim is to find a Fano $3$-fold
$V\in \ls{\OO_X(1,1,1,1)}$, invariant under the action of $G$. Then $G$ acts freely
on the surface $T:=V\cap Y$, and the quotient $T/G$ yields families of
surfaces of general type with $p_g=0$ and $K^2=3$.
\smallskip

\noindent There are only $4$ such admissible pairs. For
$G=Q_8\oplus \ZZ_2$ this has been already done in \cite{AnnSNS}, so we skip this case. Below we show that $V$ does not
exist if $G$ is $\ZZ_8\oplus \ZZ_2$ or $\ZZ_4\oplus \ZZ_4$. The remainder of this section
is concerned with describing a new irreducible component of the moduli space of canonical models of surfaces of general type with
$p_g=0$, $K^2=3$ and fundamental group $\ZZ_4\ltimes \ZZ_4$. To ease notation, let $H$ denote the class of a divisor
in $\ls{\OO_X(1,1,1,1)}$.

\begin{prop}
No divisor in $|{\mathcal O}_{X}(H)|$ is invariant by the action of $\ZZ_8 \oplus \ZZ_2$ nor by the action of  $\ZZ_4 \oplus \ZZ_4$
given in Theorem~\ref{THM:4GPS}.
\end{prop}

\begin{proof}
Note that the actions of $A$ and $B$ on $\PP^1$ naturally lift to $H^0(O_{\PP^1}(1))$, by taking
$A(x_0)=x_0$, $A(x_1)=-x_1$, $B(x_1)=x_0$, $B(x_0)=x_1$. Using (\ref{eq: action by Z_4xZ_4}) and
(\ref{eq: action of Z_8xZ_2}), these lifts induce an actions of $G$ on  $H^0(\OO_X(H))$, since
the set of $16$ vectors $x_{1i}x_{2j}x_{3k}x_{4l}$,
\mbox{$(i,j,k,l)\in \set{0,1}^4$} is a natural basis for $H^0(\OO_X(H))$.
\smallskip

\noindent
In both cases, if we denote by $g_1$ the automorphism induced
by $g$ and by $h_1$ the automorphism induced by $h$,
a direct computation shows that $g_1h_1=-h_1g_1$.
\smallskip

\noindent
Assume that there is a geometrically invariant hyperplane section $H$.
Then $H$ is the zero locus of a section $v \in H^0({\mathcal O}_X(H))$, which must be an eigenvector
for both $g_1$ and $h_1$. However, given that $g_1h_1=-h_1g_1$, this is impossible.
\end{proof}

Consider the case $(Y,\ZZ_4\ltimes \ZZ_4)$. Namely, consider the family of Calabi--Yau $3$-folds
$Y\in \LL\subset \ls{\OO_X(2H)}$, where $\LL$ is the linear system generated by the quadrics
(\ref{eq: generators of the linear system of quadrics}). Let $G=\ZZ_4\ltimes \ZZ_4$ act on $X$
as given in (\ref{eq: action by Z_4ltimesZ_4}).
Consider $V=(F_1=0)\in |{\mathcal O}_{X}(H))|$ defined by
\begin{equation}\label{eq: equation of the Fano}
F_1 := (x_{20}x_{30}-x_{21}x_{31})(x_{11}x_{40}+x_{10}x_{41})
-i(x_{20}x_{31}-x_{21}x_{30})(x_{10}x_{40}+x_{11}x_{41}).
\end{equation}

\begin{lem}\label{SingZ1}
$V$ is a Fano 3-fold polarized by  $-K_{V}=H_{|V}$. The singular locus of $V$ is the set of the points
such that $x_{i0}^2-x_{i1}^2=0$, $\forall i$ and $x_{10}x_{20}x_{30}x_{40}=-x_{11}x_{21}x_{31}x_{41}$.
Therefore, $\Sing(V) \subset \Fix(h^2)\subset \Fix(G)$, where $h$ is given in (\ref{eq: action by Z_4ltimesZ_4}).
\end{lem}
\begin{proof}
The only nontrivial claim is the statement about the singularities, which is checked locally on each of the $16$
affine open sets. In an affine open set, for example, the affine open set given by $x_{11}=x_{21}=x_{31}=x_{41}=1$,
it is easy to see that the singular points satisfy $x_{i0}^2=1$, $\forall_i$ and $x_{10}x_{20}x_{30}x_{40}= -1$.
Checking on all remaining affine open set can be done with Macaulay2.
\end{proof}

\begin{thm}\label{family}
The general element $T$ in the linear system $\LL_{|V}$ is
a simply connected smooth minimal surface of general type with
$p_g(T)=15$, $q(T)=0$ and $K_T^2=48$ on which $G$ acts freely.
The quotient $S:=T/G$ is a minimal surface of general type with
$p_g(S)=q(S)=0$, $K^2_S=3$ and fundamental group $G$.
\end{thm}
\begin{proof}
By Lemma~\ref{SingZ1} and Lemma~\ref{CYlem}, the singular locus of $V$ does not intersect the base locus of $\LL$.
Some points of the base locus of $\LL$ are contained in $V$, for example, on the affine open set
given by $x_{11}=x_{21}=x_{31}=x_{41}=1$, the point $[(0,1),(0,1),(1,1),(i,1)]$ belongs to $V\cap \operatorname{Bs}\LL$.
Nevertheless, at this point, $\frac{\partial Q_4}{\partial x_{10}}=0$, $\frac{\partial Q_4}{\partial x_{40}}=-4i$,
$\frac{\partial F_1}{\partial x_{40}}=-1$,  so $[(0,1),(0,1),(1,1),(i,1)]$ is a smooth point of $(Q_4=0) \cap V$.
A similar computation shows that $(Q_4=0) \cap V$ is smooth at all the base points of $\LL_{|V}$.
By Bertini's theorem, the general element of $\LL_{|V}$ is smooth.
The general $T$ does not contain any of the $48$ points in $\Fix(G)$, so the  $G$-action induced on $T$ is free.

\noindent
By Lefschetz hyperplane section theorem $T$ is simply connected, $q(T)=0$;
then $q(S)=0$ and \mbox{$\pi_1(S)\cong G$}. By adjunction $K_T=H_{|T}$ is ample with self-intersection \mbox{$2 H^4=48$}, and
therefore  $K_S$ is also ample and $K_S^2=48/16=3$. Moreover $p_g(T)=h^0({\mathcal O}_{X}(H))-1=15$,
and hence $\chi({\mathcal O}_S)=\chi({\mathcal O}_T)/16=\frac{1-0+15}{16}=1.$
\end{proof}

Theorem \ref{family} produces a family of minimal surfaces of general type with $p_g=0$ quotients of complete 
intersections $T=\{F=Q=0\}$ in $X$, parametrized by an open subset $P^\circ$
of the $6-$dimensional subspace $P:=\langle F_1 \rangle \times \langle Q_0,\ldots,Q_5\rangle$ of 
$P':=H^0({\mathcal O}_{X}(H)) \oplus H^0({\mathcal O}_{X}(2H))$.

\begin{lem}
The action of $G$ on $X$ lifts to an action on $P'$ such that 
$P$ is the invariant subspace $(P')^G$.
In particular, $\forall v \in P$, $T_vP'$  has an induced $G$-action making the natural identification of $P'$ 
with it $G$-equivariant.
\end{lem}

\begin{proof}
We gave the action of $G$ on $X$ in (\ref{eq: action by Z_4ltimesZ_4}) by giving a lift to 
$H^0({\mathcal O}_{X}(H))$. This determines a lift to $H^0({\mathcal O}_{X}(2H))$ by asking the $G$-equivariance of
the multiplication map 
$H^0({\mathcal O}_{X}(H)) \otimes H^0({\mathcal O}_{X}(H))\rightarrow H^0({\mathcal O}_{X}(2H))$. This lifts the 
action of $G$ to $P'$.

All vectors $v$ in $P$ are eigenvectors, thus 
the action of $G$ on $P$ is six copies of the same $1-$dimensional representation. Twisting by its dual, 
we find a lift fixing all vectors in $P$: $P\subset (P')^G$ and the natural identification of $P'$ 
with $T_vP'$ is $G$-equivariant.

To conclude we show that $\dim (P')^G = 6$. 

First we checked that the trace of the action (\ref{eq: action by Z_4ltimesZ_4}) 
on $H^0({\mathcal O}_{X}(H))$ is zero for every non-trivial element, which implies that it is isomorphic to
the regular representation $\CC[G]$. Since the regular representation is invariant by 1-dimensional twists, also the twisted representation we are 
considering is isomorphic to it, and therefore $H^0({\mathcal O}_{X}(H))^G$ has dimension $1$ and equals 
$\langle F_1 \rangle$.

On the other hand, if $(F,Q)$ is a general point in $P^\circ$,
$\OO_T(2H)\cong \OO_T(2K_T)$. 
Since the action of $G$ on $T$ has no fixed points and $T$ is a surface of general type, by the holomorphic 
Lefschetz fixed point formula
every linearization of the action on $G$ on  $H^0(2K_T)$ is isomorphic to $4$ copies of the 
regular representation; this holds in particular for the linearization induced on the cokernel by the exact sequence
$$
0 \rightarrow F_1 \otimes H^0(\OO_X(H)) \oplus \langle Q \rangle \rightarrow H^0({\mathcal O}_{X}(2H)) 
\rightarrow
H^0(\OO_T(K_T)) \rightarrow 0
$$
and therefore 
$$\dim (P')^G= \dim (F_1 \otimes H^0(\OO_X(H)) \oplus \langle Q\rangle )^G + \dim  H^0({\mathcal O}_{X}(2H))^G=2+4=6.$$
\end{proof}

We show now that all deformations of the complete intersection of an element of $|\OO_X(H)|$ with an element of $|\OO_X(2H)|$ are
obtained by moving the two hypersurfaces in their linear system.

\begin{prop}\label{all defs are embedded}
Let $Y$ be a smooth element of $|{\mathcal O}_X(2H)|$, $V$ be an element of $|{\mathcal O}_{X}(H)|$ such that $T =Y \cap V$ is smooth.
Then all small deformations of $T$ are embedded, i.e. the natural map
$$\delta \colon H^0({\mathcal O}_{V}(H)) \oplus H^0({\mathcal O}_{Y}(2H)) \rightarrow H^1(\Theta_T)$$ is surjective.
Moreover $h^1(\Theta_T)=67$ and $h^2(\Theta_T)=3$.
\end{prop}

\begin{proof}
We compute cohomology groups of sheaves on $Y$ and $T$ using (in this order) the following exact sequences:
\begin{itemize}
\item[] $0 \rightarrow {\mathcal O}_{{X}} \rightarrow {\mathcal O}_{{X}}(2H) \rightarrow {\mathcal O}_Y(2H) \rightarrow 0$, and its twist by $-H$ and by $-2H$,
\item[] $0 \rightarrow {\mathcal O}_Y \rightarrow {\mathcal O}_Y(H) \rightarrow {\mathcal O}_T(H) \rightarrow 0$,
\item[] $0 \rightarrow \Theta_{X}(-2H) \rightarrow \Theta_{X} \rightarrow \Theta_{{X}|Y} \rightarrow 0$, and its twist by $-H$,
\item[] $0 \rightarrow \Theta_Y(-H) \rightarrow \Theta_{{X}|Y} (-H) \rightarrow N_{Y|{X}}(-H) \cong {\mathcal O}_Y(H) \rightarrow 0$,
\item[] $0 \rightarrow \Theta_Y \rightarrow \Theta_{X|Y} \rightarrow {\mathcal O}_Y(2H) \rightarrow 0$,
\end{itemize}
recalling that $\Theta_{X}= {\mathcal O}_{X}(2,0,0,0)\oplus {\mathcal O}_{X}(0,2,0,0) \oplus {\mathcal O}_{X}(0,0,2,0) \oplus {\mathcal O}_{X}(0,0,0,2)$,
and, since $Y$ is a Calabi--Yau $3$-fold, that $h^0(\Theta_Y)=h^3(\Omega^1_Y)=h^{1,3}(Y)=0$. Table~\ref{cohomologies} contains the
result of this computation (empty cells are zeros).
\begin{table}[ht]
\centering\renewcommand{\arraystretch}{1.3}
\begin{tabular}{lccccc}
 & $h^0$ & $h^1$ & $h^2$ & $h^3$ & $h^4$ \\
\hline \hline
$\Theta_{X}$ & 12 &  &  & &\\ \hline
$\Theta_{X}(-H)$ &  &  &  & & \\ \hline
$\Theta_{X}(-2H)$ &  &  & & 4 & \\ \hline
$\Theta_{X}(-3H)$ &  &  &  & & \\ \hline
${\mathcal O}_Y(2H)$ & 80  &  &  & & \\ \hline
 ${\mathcal O}_Y(H)$& 16 &  &  & & \\ \hline
${\mathcal O}_Y$ & 1 &  &  &1 & \\ \hline
${\mathcal O}_T(H)\cong \omega_T$ & 15 &  & 1 & & \\ \hline
$\Theta_{X|Y}$ &  12 &  & 4 & & \\ \hline
$\Theta_{X|Y}(-H)$ &  &  &  & & \\ \hline
$\Theta_Y(-H)$ &  & 16 &  &  & \\ \hline
$\Theta_Y$ &  & 68 & 4 &  & \\ \hline
$\Theta_Y$ &  & 68 & 4 &  & \\ \hline
&  &  &  &  & \\
\end{tabular}
\caption{Cohomology table}\label{cohomologies}
\end{table}

\noindent
Now, consider the further exact sequence
$$0 \rightarrow \Theta_Y(-H) \rightarrow \Theta_Y \rightarrow \Theta_{Y|T} \rightarrow 0.$$
Then $h^1(\Theta_{Y|T})-h^0(\Theta_{Y|T})=52$ and $h^2(\Theta_{Y|T})=4$.
Finally, from
\begin{equation}\label{last sequence}
0 \rightarrow \Theta_T \rightarrow \Theta_{Y|T} \rightarrow N_{Z|T}={\mathcal O}_T(H) \rightarrow 0,
\end{equation}
recalling that since $T$ is of general type, $h^0(\Theta_T)=0$, we get $h^1(\Theta_T)=67$,  \mbox{$h^2(\Theta_T)=3$}.
\medskip

\noindent
Consider the map, $\delta$, in the statement. It has $2$ components.
The first component is the composition of the restriction map $H^0({\mathcal O}_{V}(H)) \rightarrow H^0({\mathcal O}_{T}(H))$,
which is surjective, with the coboundary map $H^0({\mathcal O}_T(H))\rightarrow H^1(\Theta_T)$ in the cohomology sequence
of (\ref{last sequence}), whose cokernel surjects onto $H^1(\Theta_{Y|T})$. Hence the image of $\delta$ contains the
kernel of the map $H^1(\Theta_T) \rightarrow H^1(\Theta_{Y|T})$.
The composition of the second component with the map $H^1(\Theta_T) \rightarrow H^1(\Theta_{Y|T})$
factors as $H^0({\mathcal O}_{Y}(2H))\rightarrow H^1(\Theta_Y) \rightarrow H^1(\Theta_{Y|T})$ which is a composition
of surjective maps, so surjective. It follows that $\delta$ is surjective.
\end{proof}

We can finally prove our last result.
\begin{thm} \label{comp1} The family of surfaces in Theorem~\ref{family} dominates an
irreducible component of dimension $4$ of the moduli space of minimal surfaces of general type of \mbox{genus $0$}.
\end{thm}
\begin{proof}
Take a smooth surface $T \in \LL_{|V}$ fulfilling the conditions of Theorem \ref{family} and 
consider the exact sequence 
$$
0 \rightarrow \Theta_T \rightarrow \Theta_{X|T} \rightarrow \OO_T(H) \oplus \OO_T(2H) \rightarrow 0.
$$
The action of $G$ on $X$ and $T$ induce actions on $\Theta_T$ and $\Theta_{X|T}$ for which $\Theta_T \rightarrow \Theta_{X|T}$ is 
$G$-equivariant, and therefore induces an action on the cokernel, 
making the coboundary map 
$$d_T \colon H^0( \OO_T(H) \oplus \OO_T(2H)) \rightarrow H^1(\Theta_T),$$ which is the differential of the map 
sending embedded deformations of $T$ to abstract deformations, $G$-equivariant. Since   
$\delta$ factors through $d$, $d$ is surjective and so the induced map among the $G$-invariant subspaces
\begin{equation}\label{d^G}
d_T^G \colon H^0( \OO_T(H) \oplus \OO_T(2H))^G \rightarrow H^1(\Theta_T)^G
\end{equation}
is surjective too.

The \'etale map $\pi \colon T \rightarrow S=T/G$ induces by pull-back an isomorphism
$\Theta_S \rightarrow (\pi_* \Theta_T)^G$. Since $G$ is finite, $\pi_* \Theta_T$ splits 
as direct sum of invariant subsheaves, one for each character of $G$; in particular inducing an isomorphism from
$H^1(\Theta_S)$ to $H^1(\Theta_T)^G$. 

We have then a commutative diagram
\begin{equation*}
\xymatrix{
T_vP\ar[d]\ar^{d_v}[rr]&&H^1(\Theta_S)\ar[d]\\
(T_vP')^G \ar[r]&(H^0( \OO_T(H) \oplus \OO_T(2H)))^G \ar[r]&H^1(\Theta_T)^G
}
\end{equation*}
where $v \in P^\circ$ is a point corresponding to a surface  $S=T/G$, $d_v$ is the differential at $v$ of the map from the family in Theorem 
\ref{family} to the abstract deformations of $S$. Since (recall Proposition \ref{all defs are embedded}) 
both vertical maps are isomorphisms, and the two horizontal maps at the bottom are surjective,
$d_v$ is surjective too. This shows that the family dominates a component of the moduli space.

Note that $P^\circ$ has dimension $6$ and multiplying the two equations by constants gives
a faithful $\CC^* \times \CC^*$ action on $P^{\circ}$ trivial on the moduli space, so this component has dimension at most $6-2=4$.
On the other hand, the expected dimension is $10 \chi-2K_S^2=10-6=4$; hence its dimension is $4$.
\end{proof}

\begin{remark}\label{comp2} A similar
argument shows that also the family given by the action of $\ZZ_2 \oplus Q_8$
dominates an irreducible component of dimension $4$ of the moduli space of minimal surfaces of general type of genus $0$.
As explained by the last two authors in \cite{AnnSNS}, this was proved by Bauer and Catanese in \cite{Burniat3}, where an open set of that family is constructed and studied with a different method. Therefore we decided not to give here the details of our proof of that case.
\end{remark}

{\bf Acknowledgments}. The authors are grateful to Gian Pietro Pirola for helpful remarks, and to an anonymous 
referee for improving some arguments of the last section.

\end{document}